\documentclass[11pt]{article}
\usepackage{fullpage}
\usepackage{amssymb}
\usepackage{amsmath}
\usepackage{amsthm}

\usepackage{amssymb,latexsym}
\usepackage{amsmath}
\usepackage{amsthm,amsfonts}
\usepackage{dsfont}  

\usepackage{amsmath,amssymb,latexsym,amsthm}
\usepackage[pdftex]{graphicx}
\usepackage{graphics}
\usepackage{psfrag}
\usepackage{color, epsfig}
\usepackage{mathrsfs}
\usepackage[english]{babel}   
\usepackage{hyperref}   

\newtheorem{theorem}{Theorem}
\newtheorem{problem}{Problem}
\newtheorem{remark}{Remark} 
\newtheorem{lemma}{Lemma}
\numberwithin{equation}{section}

\usepackage{amsfonts}
\usepackage{latexsym}
\usepackage{amsmath}
\usepackage{amssymb}
\usepackage[mathscr]{eucal}
\usepackage{graphicx,color}
\usepackage{epstopdf}
\usepackage[latin1]{inputenc}

\newcommand\eps{\varepsilon}

\begin{document}

\title{Event-triggered boundary damping of a linear wave equation}

\author{Lucie Baudouin\thanks{LAAS-CNRS, Universite de Toulouse, CNRS, Toulouse, France.
E-mail: {\tt baudouin@laas.fr}, {\tt tarbour@laas.fr}.}
\and
Swann Marx\thanks{LS2N, Ecole Centrale de Nantes and CNRS UMR 6004, CNRS, Nantes, France. 
E-mail: {\tt swann.marx@ls2n.fr}.}
\and
Sophie Tarbouriech$^*$ 
\and
Julie Valein\thanks{Université de Lorraine, CNRS, Inria, IECL, F-54000 Nancy, France. 
E-mail: {\tt julie.valein@univ-lorraine.fr}.}
}

\maketitle

\begin{abstract}This article presents an analysis of the stabilization of a multidimensional partial differential wave equation under a well designed event-triggering mechanism that samples the boundary control input. 
The wave equation is set in a bounded domain and  the control is performed through a boundary classical damping term, where the Neumann boundary condition is made proportional to the velocity. First of all, existence and regularity of the solution to the closed-loop system under the event-triggering mechanism of the control are proven. Then, sufficient conditions based on the use of a specific Lyapunov functional are proposed in order to ensure that the solutions converge into a compact set containing the origin, that can be tuned by the designer. Furthermore, as expected, any Zeno behavior of the closed-loop system is avoided.
\end{abstract}


{\bf Acknowledgments}: This research was partially funded by the French grants ANR-19-CE48-0004-01 (ODISSE),  ANR-20-CE40-0009 (TRECOS) and by the  CIMI LabEx, ANR-11-LABX- 0040.

\section{Introduction}

The boundary stabilization of a  multidimensional linear wave equation by means of a continuous feedback law is nowadays quite well-known as illustrated by \cite{komornik1990direct}, \cite{komornik1991rapid}, \cite{komornik1997exact},  \cite[Section 7.6]{tucsnak2009} or \cite{Haraux}. We may also mention the case where this feedback is subject to delay (see for instance \cite{nicaise2011exponential}).

It is then interesting to address the digital implementation of the control law for this kind of systems, or equivalently to consider the case where an existing continuous feedback is modified via a sample-and-hold mechanism which follows an aperiodic law. To be more specific, the latter feedback  follows an \emph{event-triggered} control law. 
This question, first developed in the finite-dimensional context (see e.g., \cite{pos:tab:nes:ant/IEEE2015,tabuada2007event}), is closely related to the implementation of such feedback laws: it might be indeed numerically heavy to modify the feedback at each instant, and one is rather interested in updating it only at some instants. 
Moreover, one wants to choose these instants accordingly with events that might deteriorate the stability of the system. This justifies that the updates instants are chosen aperiodically. 
However, from a theoretical viewpoint, one of the main issues is related to the potential occurrence of \emph{Zeno phenomenon}, corresponding to trigger  infinitely many updates of the control in a bounded time interval, or equivalently to the fact that the sequence of the updates instants may accumulate on a final time horizon. 

In addition to the finite dimensional setting,  event-triggering mechanisms have been applied more recently to some partial differential equations. To the best of our knowledge, this first application dates back to \cite{espitia2016event,esp:tan:tar/CDC2017}, and it concerns one-dimensional systems of hyperbolic equations. In the same setting, we may also mention \cite{davo2018sample} which treats the case of linear hyperbolic systems subject to sampled-data controllers with periodic  updates. It is also instructive to cite \cite{espitia2021event} concerning reaction-diffusion PDEs. Finally, some periodic event-triggered control has been applied to abstract systems in \cite{wakaiki2019stability} with a semigroup viewpoint.

Except this latter contribution, most of the systems that have been considered until now are one-dimensional in the space variable. However, more recently, in \cite{koudohode_Automatica} and \cite{koudohode_ECC}, a multi-dimensional wave or Schrödinger equation have been studied from the event-triggered control perspective. In contrast with the boundary feedback law that we are going to study here, the one in \cite{koudohode_Automatica} and \cite{koudohode_ECC} is distributed in the space domain, and this allows  in particular to apply some finite-dimensional strategy to prove the absence of Zeno phenomenon. 

As already mentioned, we are interested in the boundary feedback law case, which corresponds to an unbounded control operator, leading thus to specific difficulties to overcome. To the best of our knowledge, it is the first time that such a strategy is applied to a multi-dimensional PDE. Indeed, in contrast with \cite{wakaiki2019stability}, we allow the control to be \emph{aperiodic}. 
However, due to the lack of regularity when considering such settings, we are not able to state an asymptotic stability result, but rather the convergence of the solutions into a compact set. Actually, sufficient conditions based on the use of an appropriate Lyapunov functional are proposed in order to ensure that the solutions converge into a compact set containing the origin with tunable shape, which corresponds in the automatic control theory to a \emph{practical stability}.
This aspect constitutes a salient point with respect to \cite{koudohode_Automatica}, in which we studied he wave equation under an event-triggering mechanism that updates a distributed damping source term. \\

This paper is organized as follows. Section~\ref{sec_state} introduces the main problem together with the questions (in a mathematical way) that structures this paper. The main results  of the paper and their proofs are collected in Section~\ref{sec_main}. Section~\ref{sec_num} provides some numerical illustrations of the efficiency of our feedback law and some concluding remarks and further research lines. \\
  
{\bf \noindent Notation: } The set $\Omega$ is an open bounded domain of $\mathbb R^n$ of boundary $\partial\Omega$. The partial derivative of a function $z$ with respect to $x_i$ is denoted $\partial_{x_i} z = \frac{\partial z}{\partial x_i}$, the gradient is the vector function $\nabla z = (\partial_{x_i} z)_{i=1,...,n}$ and its normal component is $\partial_\nu z = \nabla z \cdot \nu$, $\nu$ being the unit normal outward vector to $\Omega$. 
The Hilbert space of square integrable functions over $\Omega$ with values in $\mathbb R$ is denoted by $L^2(\Omega)$ with scalar product $\langle z_1,z_2 \rangle = \int_\Omega z_1(x)z_2(x) dx$, and associated norm $\|z\| = \sqrt{\langle z,z \rangle}$. The Sobolev space  $\big\{ z\in L^2(\Omega),  \nabla z \in \left(L^2(\Omega)\right)^N \big\}$  is  denoted $H^1(\Omega)$, with norm $\|z\|^2_{H^1} = \|z\|^2 + \|\nabla z\|^2$. 
The Poincar\'e constant of $\Omega$ is denoted $C_\Omega$ (see Lemma~\ref{Poincare} in appendix).
For the boundary-valued functional space $L^2(\partial\Omega)$, the norm's notation will be kept obvious as $\|z\|^2_{L^2(\partial\Omega)} = \int_{\partial\Omega} \vert z(x) \vert ^2 d\sigma $ .
Finally, the Laplacian operator writes $\Delta z =  \sum_{n=1}^{N}\partial_{x_n}^2 z $. In this article, $z = z(t,x)$ will be the infinite dimensional state of an evolution partial differential equation.

\section{Problem Statement}
\label{sec_state}

  
We are interested in a multi-dimensional wave equation set in a bounded domain $\Omega \subset \mathbb R^n$ and controlled through a part $\Gamma_1$ of the boundary $\partial\Omega$. The partial differential system under consideration is described by
\begin{equation}\label{EqU}
	\left\{
	\begin{array}{ll}
		\partial_t^2z(t,x) - \Delta z(t,x) = 0,\quad& (t,x)\in\mathbb{R}_+\times \Omega\\
		z(t,x) = 0, \quad& (t,x)\in\mathbb{R}_+\times \Gamma_0\\
		\partial_\nu z(t,x) = u(t,x), \quad&  (t,x)\in\mathbb{R}_+\times \Gamma_1\\
		z(0,x)=z_0(x),&  x\in \Omega,\\
		 \partial_t z(0,x)=z_1(x),\quad &  x\in \Omega,
	\end{array}
	\right.
\end{equation}
where $z$ and $u$ denote the state and the control variables, respectively. 
The Dirichlet boundary datum of the unknown $z$ is homogeneous on a non-empty part $\Gamma_0$ of $\partial\Omega$ and $u$ is a Neumann boundary control acting on the complementary part $\Gamma_1= \partial\Omega \setminus \Gamma_0$.
 
 
It is well known (see e.g., \cite{tucsnak2009}) that if $u \in L^1(\mathbb R_+ ;L^2(\Gamma_1))$ and 
 $(z_0,z_1)\in H^1(\Omega)\times L^2(\Omega)$, and assuming the appropriate compatibility conditions on $\{0\}\times\partial\Omega$ between them, then the initial and boundary value problem \eqref{EqU} is well-posed and has a unique solution with the following regularity
$$
	 z \in C^0(\mathbb R_+; H^1_{\Gamma_0}(\Omega)) \cap C^1(\mathbb R_+; L^2(\Omega))
$$ 
where
$$
H^1_{\Gamma_0}(\Omega) = \{ v\in H^1(\Omega), v\vert_{\Gamma_0} = 0 \}.
$$

Borrowing from the introduction of \cite{komornik1991rapid}, recall that the most general result in finding large classes of boundary feedbacks giving exponential decay has been proved by Bardos, Lebeau and Rauch in \cite{Bardos}, where they characterized a geometric control condition about $\Gamma_1$ from the rays of optical rules. Nevertheless, their method does not provide explicit decay rate estimates. On the other hand, we would like to explore here some special feedbacks giving exponentially fast energy decay and study how a well chosen event-triggered feedback law can maintain such a quality.

Actually, we consider that the control input will take the following shape
\begin{equation}\label{eq:ucontinuous}
u(t,x):=-\alpha(x)\partial_t z(t,x)
\end{equation}
where the boundary damping weight $\alpha \in L^\infty(\Omega)$ writes
\begin{equation}\label{alpha}
\alpha (x) =\alpha_1 (x-x_0)\cdot \nu(x)\geq \alpha_0 >0 , \quad \forall x\in \Gamma_1,
\end{equation}
with $\alpha_1>0$ a tuning coefficient and $\nu$ the unit normal outward vector to $\Omega$.
One should note here that the control operator is thus unbounded, since the control only act on the boundary of the domain $\Omega$.\\
Nevertheless, if $\Gamma_1$ satisfies the following  geometric condition, 
for $x_0 \notin \overline\Omega\subset \mathbb R^n$,
\begin{equation}\label{Gamma1}
\Gamma_1 = \{ x\in \partial\Omega, (x-x_0)\cdot \nu(x) > 0\}
\end{equation}
and if $\Gamma_1$ and $\Gamma_0 = \partial\Omega \setminus \Gamma_1$ are such that $\Gamma_0 \cap \Gamma_1 = \emptyset$, then  the origin of the closed-loop system is globally exponentially stable (see  \cite{komornik1991rapid}).
For previous controllability results in the same kind of setting, see for example, \cite{Lions}, \cite{LasieckaLionsTriggiani} (or read \cite{Haraux}, \cite{tucsnak2009}).
\begin{remark}
The assumptions on the boundary parts $\Gamma_0$ and $\Gamma_1$ of $\partial\Omega$ mean roughly that $\Omega$ is a domain with a hole and $x_0$ is taken in this hole so that $\Gamma_0$ is the inner boundary and  $\Gamma_1$ the outer boundary. The main reason for such a ``non-touching" constraint lies in the $H^2(\Omega)$ regularity we will need to build our control loop.
\end{remark}
In this article, the objective relies on the way to implement the control input $u=-\alpha\partial_t z$, under a sample-and-hold mechanism, in order to keep the best possible stabilization result. The idea is indeed that the control applied on $\Gamma_1$ will be only updated at certain instants $\{t_k\}_{k\in\mathbb N}$ and held constant between two successive sampling instants. Besides, the sampling instants, which form an increasing sequence, will not be periodically chosen but will follow a specifically designed event-triggering rule given below. 

More precisely, the control input we mean to apply can be written as 
\begin{equation}
	u(t,x) = -\alpha(x)\partial_t z(t_k,x),\quad \forall (t,x)\in [t_k,t_{k+1}[\times \Gamma_1 
\end{equation}
leading to the following closed-loop system:
\begin{equation}\label{Eqet}
	\left\{
	\begin{array}{ll}
		\partial_t^2z(t,x) - \Delta z(t,x) = 0,& \hbox{in }\mathbb{R}_+\times \Omega\\
		z(t,x) = 0, & \hbox{on }\mathbb{R}_+\times \Gamma_0\\
		\partial_\nu z(t,x) = -\alpha(x)\partial_t z(t_k,x),  &  \hbox{on }[t_k,t_{k+1}[\times \Gamma_1, \quad  \forall k\in\mathbb N \\
		z(0,x)=z_0(x),~\partial_t z(0,x)=z_1(x),&    \hbox{in } \Omega.
	\end{array}
	\right.
\end{equation}
Before detailing the triggering law, let us define the natural energy (sum of the kinetic and potential energies) of the wave equation by
\begin{equation}\label{energy}
	\begin{aligned}
		E(t) = E(z(t), \partial_t z(t))&:=  \dfrac{1}{2}  \|\nabla z(t)\|^2 + \dfrac{1}{2}  \|\partial_t z(t)\|^2,
	\end{aligned}
\end{equation}
noting that the closed-loop system's state has actually two components $ \begin{pmatrix} z \\\partial_t z \end{pmatrix}$.\\
Now, we can state the problem we will analyse in this article.
\begin{problem} \label{pb1}
Considering the simplified event-triggering law:
\begin{equation}\label{et-law}
		t_{k+1}:=\inf\Big\{ t\geq t_k, ~\|\partial_t z(t) - \partial_t z(t_k)\|_{L^2(\Gamma_1)}^2 
		- \gamma E(t)  - \nu_0 \geq 0\Big\},
\end{equation}
we want to design the positive parameters $\gamma$ and  $\nu_0$  in order to guarantee:
\begin{itemize}
	\item[$(i)$] the well-posedness of the closed-loop system \eqref{Eqet}-\eqref{et-law},
	\item[$(ii)$] the global convergence of the closed-loop solutions towards a compact set $\mathcal{A}$, which corresponds to an outer-estimation of the globally stable attractor,
	\item[$(iii)$] the absence of Zeno behavior.
\end{itemize}
\end{problem}
The sampling law is inspired by \cite{koudohode_Automatica}, where the control operator is bounded, but the convergence theorems it allows to prove here (see next section) is subtly different from the exponential stability result proved in \cite{koudohode_Automatica}. We can explain in a nutshell that the term $\nu_0$  has the  goal of guaranteeing the absence of Zeno behavior that cannot be obtained otherwise, because of the unbounded nature of the control operator. This Zeno phenomenon arises when an infinite sequence of updates $t_k$ can be triggered in finite time. The $\nu_0$ term proves the absence of accumulation points in the update sequence $(t_k)_{k\in \mathbb N}$. Nevertheless, it allows attractivity of a neighborhood of the origin (as one will read below in the main result section, see also  \cite{ber:lak/TMJ1980}, \cite{kha}, \cite{pos:tab:nes:ant/IEEE2015}), and even exponential convergence towards the global attractor, but prevents the system from achieving exponential stability of the origin as in \cite{koudohode_Automatica}.\\

To go further into the presentation of our work, one should notice that actually, in the continuous setting \eqref{EqU}-\eqref{eq:ucontinuous}, the energy $E$ defined in \eqref{energy} acts as a weak Lyapunov functional. It is only a part of the (strict) Lyapunov functional we will need in the definition of our attractor and use in the proof of our convergence results. 
But indeed, computing formally the time-derivative of $E$ brings after some integration by parts,
$$\dot E(t) = \int_{\Gamma_1}u(t,x) \partial_t z(t,x)\,d\sigma = - \int_{\Gamma_1} \alpha(x) \left\vert \partial_t z(t,x)\right\vert^2\,d\sigma \leq 0, \quad \forall t\geq0.$$
This is why we define here a Lyapunov functional candidate, with $\eps>0$, by
\begin{equation}\label{Lyap}
V(t) = V(z(t),\partial_t z(t))
:= E(t)
+ \eps\int_{\Omega} \left( 2(x-x_0)\cdot \nabla z(t,x) + (n-1)z(t,x) \right) \partial_t z(t,x) \,dx,
\end{equation}
aiming at exhibiting conditions under which we have $\dot V(t)\leq -2 \delta V(t)$, $ \forall t\geq0$. This choice of functional takes root in the article \cite{komornik1990direct} and one can also have a look to \cite{tucsnak2009} for details and related references.

Finally, since the sampling law proposed in \eqref{et-law} is aperiodic, we will have to carefully check whether it does, or not, produce any Zeno phenomenon. And in that perspective, let us introduce
 the maximal time $T^*$ under which the closed-loop system subjected to this event-triggering rule has a solution:
\begin{equation}
\label{T}
\left\{
\begin{array}{ll}
T^*=+\infty  &  \text{ if $(t_k)$ is a finite sequence, }\\
T^*=\displaystyle\limsup_{k\to +\infty} t_k &  \text{ if not}.
\end{array}
\right.
\end{equation}
Hence, later in the article, the proof of the absence of Zeno behavior will actually result from the proof that $T=+\infty $, since no accumulation point of the sequence $(t_k)_{k\geq 0}$ will be possible.

\section{Main Results}
\label{sec_main}
In this section, we  first present the core result of our contribution, providing an answer to item $(ii)$ of  Problem \ref{pb1}. 
Then we address $(i)$ and $(iii)$ which are indeed closely related. Finally, the proof of the theorem guaranteeing $(ii)$ is given.

\subsection{Stability and convergence theorem}
\label{sec_stab}
Let us state the main result of this article describing the conditions for the global exponential convergence of the trajectories to an attractor.


\begin{theorem}\label{Att}
Let $\Omega \subset \mathbb R^n$, $R = \max_{x\in\Gamma_1} \vert x-x_0 \vert $, $\alpha_1>0$ be the boundary damping coefficient and $C_\Omega$ the Poincar\'e constant of domain $\Omega$ (see appendix). Let the event-triggered control law \eqref{et-law} be defined for some tuning parameters $\gamma>0$ and $\nu_0  >0$.
Assume that there exist coefficients $\lambda_1> 0$, $\lambda_2 > 0$ and $0<\varepsilon <  1/(2R + (n-1)C_\Omega)$ such that 
\begin{equation} \label{eq:pscond2a}
  -\eps  +\dfrac{\lambda_1 \gamma}2 + \lambda_2 C_\Omega(R + nC_\Omega) < 0
 \end{equation}
and the following linear matrix inequality  holds
 \begin{equation} \label{eq:pscond1a}
 M := \left(\begin{array}{cccc}
 -\lambda_2 & 0 & -  (n-1)\alpha_1\eps/2 & (n-1)\alpha_1\eps/2 \\
 * & - \eps/R^2 & - \alpha_1\eps &  \alpha_1\eps \\
 * & * & \eps-\alpha_1  &  \alpha_1/2 \\
 * & * & * &  -\lambda_1/R
\end{array}\right) \prec 0.
 \end{equation}
  Then for any initial state $(z_0,z_1)\in H^2(\Omega)\cap H^1_{\Gamma_0}(\Omega)\times H^1_{\Gamma_0}(\Omega)$,  
the closed-loop trajectories $ \begin{pmatrix} z \\\partial_t z \end{pmatrix}  \in C(\mathbb R_+; H^2(\Omega)\cap H^1_{\Gamma_0}(\Omega)) \times C(\mathbb R_+; H^1_{\Gamma_0}(\Omega)) $ of system~\eqref{Eqet} under the event triggering rule \eqref{et-law} 
converge exponentially fast into the attractor 	
$$
\mathcal{A} = \left\{\begin{pmatrix} v \\w \end{pmatrix} \in  H^1_{\Gamma_0}(\Omega)\times L^2(\Omega) \hbox{ such that } V(v,w) < r \right\} 
$$
with 
$$
V(v,w) = \dfrac{1}{2}  \|v\|^2 + \dfrac{1}{2}  \|w\|^2
+ \eps \big ( \langle 2(x-x_0)\cdot \nabla v , w \rangle + \langle (n-1)v , w \rangle \big),
$$
a radius 
$$
r = \lambda_1 \nu_0 \left( \dfrac{2\eps - \lambda_1 \gamma  - 2 \lambda_2 C_\Omega(R + nC_\Omega)}{(1+\varepsilon (2R+(n-1)C_\Omega))}-2 \delta\right)^{-1}
$$
and an exponential decay rate 
\begin{equation}\label{delta} \delta <   \dfrac{\eps - \dfrac{\lambda_1 \gamma}2 - \lambda_2 C_\Omega(R + nC_\Omega)}{1+\varepsilon (2R+(n-1)C_\Omega)}.
\end{equation}
Then the item (ii) 
of Problem \ref{pb1} is solved.
\end{theorem}

In a way, this theorem proves that the autonomous closed-loop system  \eqref{Eqet}-\eqref{et-law} is globally practically stable, as any solution of the system converges towards the attractor $\mathcal{A}$. Note that the considered ball can be arbitrarily small but different from the origin.

\begin{remark} If one wants to replace the term $\nu_0$ in the event-triggering rule \eqref{et-law} by a term $\epsilon_0 e^{-2\theta t} $ as in \cite{bau:mar:tar/cpe2019}, one should notice first that the associated closed-loop system is not autonomous anymore. But if we set  
\begin{equation}\label{nzo}
0< \epsilon_0 < E(0) = \|(z_0,z_1)\|_{H^1_{\Gamma_0}(\Omega)\times L^2(\Omega)}
\end{equation} 
and choose $\delta<\theta$, an energy estimate $E(t) \leq K E(0) e^{-2\delta t}$ can be proved, yielding the exponential decay at rate $\delta$ of the system's energy towards its equilibrium point, as soon there exist coefficients $\lambda_1> 0$, $\lambda_2 > 0$ and $0<\varepsilon <  1/(2R + (n-1)C_\Omega)$  such that the linear matrix inequality \eqref{eq:pscond1a} holds, together with
\begin{equation} \label{eq:pscond2}
  -\eps +\delta +\dfrac{\lambda_1 \gamma}2 + \lambda_2 C_\Omega(R + nC_\Omega) < 0.
 \end{equation}
Nevertheless, this non-zero initial energy condition \eqref{nzo} brings a questionable situation that do not allow to talk about global exponential stability, without bringing a result that could be defined as local neither. Actually, this lack of uniformity of the exponential stability result with respect to the initial data is the precise reason why we state Theorem \ref{Att} instead of this convergence result.
\end{remark}

\begin{remark}
Comparing the setting presented in \eqref{et-law} to the more simple one that can be proposed for a wave equation with in-domain sampled damping (see \cite{koudohode_Automatica}), where no extra terms $\nu_0$ or $ \epsilon_0 e^{-2\theta t}$ are needed, one should be aware that the point is only concerning the Zeno behavior's avoidance. In the in-domain damping case, there is a lemma that proves a non-vanishing property for the system's energy, so that we do not need the extra terms. Such a lemma cannot be proved when a boundary control is at stake, because of its unbounded nature.
\end{remark}

\subsection{Well-posedness of the closed-loop systems}
\label{sec_wp}
The following result guarantees that items $(i)$ and $(iii)$ of Problem \ref{pb1} 
both hold.

\begin{theorem}
\label{WP}
\label{thm-wp}
Consider the linear wave equation \eqref{Eqet} under the event-triggering mechanism \eqref{et-law}.
For any initial condition $(z_0,z_1)\in H^2(\Omega)\cap H^1_{\Gamma_0}(\Omega)\times H^1_{\Gamma_0}(\Omega)$, there exists a unique solution 
	$$ z \in C^0(\mathbb R_+; H^2(\Omega)\cap H^1_{\Gamma_0}(\Omega)) \cap C^1(\mathbb R_+; H^1_{\Gamma_0}(\Omega)) $$
and  the Zeno phenomenon is avoided. 
\end{theorem}

Before proving the theorem, it is worth noticing that we are only able to prove this result for strong solutions. This is mainly due to the fact that we are using trace theorems, that need strong regularity on the solution.

\begin{proof} - 
The proof of Theorem~\ref{WP} is divided into three steps. We first prove that the closed-loop system \eqref{Eqet} is well-posed on every sample interval $[t_k,t_{k+1}]$ for the event trigered law \eqref{et-law}, in a way such that one obtains a unique solution $z\in C([0,T^*),H^2(\Omega)\cap H^1_{\Gamma_0}(\Omega))\cap C^1([0,T^*); H^1_{\Gamma_0}(\Omega))$. 
Then, we show that \eqref{Eqet}-\eqref{et-law}  avoid the Zeno phenomenon. Finally, gathering these information proves that the solution $z(t,x)$ exists for any $(t,x)$ in $\mathbb R_+\times\Omega$. \\

\noindent $\bullet$ {\it Existence, uniqueness and regularity of the solution:} \\
We proceed by induction. First, let us focus on the initialization interval $[0,t_1]$, and prove that the solution belongs to the awaited functional space and is unique. Then, we will assume that for a fixed integer $k$ the regularity holds true up to $t_{k+1}$, and proceed on the next time interval, identically as on $[0,t_1]$.

\noindent {\it \underline{Initialization.}} On the first time interval, \eqref{Eqet} reads
\begin{equation}\label{eq:zon0t1}
\left\{
\begin{array}{lr}
\partial_t^2 z(t,x) - \Delta z(t,x) = 0,& (t,x)\in [0,t_1]\times \Omega\\
 z(t,x) = 0,&  (t,x)\in [0,t_1]\times \Gamma_0\\
\partial_\nu z(t,x)=-\alpha(x)z_1(x), & (t,x)\in [0,t_1]\times \Gamma_1\\
 z(0,x) = z_0(x),\: \partial_t z(0,x) = z_1(x),& x\in \Omega.
\end{array}
\right.
\end{equation}
This is a wave equation with non homogeneous boundary condition. 
By assumption, $(z_0,z_1)\in H^2(\Omega)\cap H^1_{\Gamma_0}(\Omega)\times H^1_{\Gamma_0}(\Omega)$. Since $H^1_{\Gamma_0}(\Omega)\hookrightarrow H^{1/2}(\Gamma_1)$ (details if needed can be found in \cite[Section 13.6]{tucsnak2009}) and $\alpha \in L^\infty(\Gamma_1)$, one has $-\alpha z_1\vert_{\Gamma_1}\in H^{1/2}(\Gamma_1)$.

\begin{itemize}
\item[-] On the one hand, consider $y_s$ solution to stationary problem
\begin{equation}\label{eq:ystat}
\left\{
\begin{array}{ll}
- \Delta y_s(x) = 0,& x\in \Omega\\
 y_s(x) = 0,&  x\in \Gamma_0\\
\partial_\nu y_s(x)=-\alpha(x)z_1(x), \quad& x\in \Gamma_1.
\end{array}
\right.
\end{equation}
On the other hand, define $\widetilde y$ such that $\partial_{\nu}\widetilde y = -\alpha(x)z_1(x)\in H^{1/2}(\Omega)$. By the trace theorem (see for instance \cite[Theorem 9.4]{LionsMagenes}), $\widetilde y \in H^2(\Omega)$. Now define a cut-off function $\eta\in C^{\infty}(\Omega)$ such that $\eta=1$ on $\Gamma_1$ and $\eta=0$ outside a neighbourhood of $\Gamma_1$
and set $w=y_s-\eta\widetilde y$. Then $w$ satisfies
$$
\left\{
\begin{array}{ll}
- \Delta w(x) = g,\quad& x\in \Omega\\
 w(x) = 0,&  x\in \Gamma_0\\
\partial_\nu w(x)=0, & x\in \Gamma_1,
\end{array}
\right.
$$
where $g=(\Delta\eta) \widetilde y+2\nabla\eta\cdot\nabla\widetilde y+\eta\Delta\widetilde y\in L^2(\Omega)$. Using \cite[Theorem 1 in Section 6.3.1.]{evansPDE}, one has $w\in H^2(\Omega)$. Consequently $y_s=w+\eta\widetilde y\in H^2(\Omega)$. 

\item[-]  Consider now $y_e$ solution to the evolution equation
\begin{equation}\label{eq:yev}
\left\{
\begin{array}{ll}
\partial_t^2 y_e(t,x) - \Delta y_e(t,x) = 0,\quad& (t,x)\in [0,t_1]\times \Omega\\
 y_e(t,x) = 0,&  (t,x)\in [0,t_1]\times \Gamma_0\\
\partial_\nu y_e(t,x)=0, & (t,x)\in [0,t_1]\times \Gamma_1\\
 y_e(0,x) = z_0(x)-y_s(x),& x\in \Omega\\
\partial_t z(0,x) = z_1(x),& x\in \Omega.
\end{array}
\right.
\end{equation}
As $(z_0-y_s,z_1)\in H^2(\Omega)\cap H^1_{\Gamma_0}(\Omega)\times H^1_{\Gamma_0}(\Omega)$, using the Lumer-Phillips theorem, one can show (see for instance \cite[Section 3.9]{tucsnak2009} with $b=0$) that system \eqref{eq:yev} admits a unique solution
$$y_e\in C^0([0,t_1]; H^2(\Omega)\cap H^1_{\Gamma_0}(\Omega)) \cap C^1([0,t_1]; H^1_{\Gamma_0}(\Omega)) :=\mathcal H_0.
$$

\item[-]  Consequently $z=y_s+y_e$ solution to \eqref{eq:zon0t1} satisfies also $z\in \mathcal H_0$.\\
And since $\partial_t z=\partial_t y_e$, we can deduce that $\partial_t z(t_1,\cdot)\in H^1_{\Gamma_0}(\Omega)$ and so $\partial_t z(t_1,\cdot)\vert_{\Gamma_1}\in H^{1/2}(\Gamma_1)$.
\end{itemize}
\begin{remark}
The assumption  $\Gamma_0 \cap \Gamma_1 = \emptyset$ is a strong necessity linked with the $H^2(\Omega)$ regularity we seek for the solution $z$.
\end{remark}

\noindent {\it \underline{Heredity.}} Fix $k\in \mathbb{N}$ and assume that 
$$z \in C([t_k,t_{k+1}];H^2(\Omega)\cap H^1_{\Gamma_0}(\Omega))\cap C^1([t_k,t_{k+1}];H^1_{\Gamma_0}(\Omega)):=\mathcal H_k.$$ 
Denote by $z_{2k+2}$ and  $z_{2k+3}$ the position and velocity function values of the wave at time $t_{k+1}$ (compatibility conditions are met by definition).
Now consider \eqref{Eqet} over the time interval $[t_{k+1},t_{k+2}]$:
$$
\left\{
\begin{aligned}
&\partial_t^2 z(t,x) - \Delta z(t,x) = 0 ,\qquad\qquad~ \hbox{in } [t_{k+1},t_{k+2}]\times \Omega\\
& z(t,x) = 0,\qquad\qquad\qquad\qquad \qquad  ~\hbox{on } [t_{k+1},t_{k+2}]\times \Gamma_0\\
& \partial_\nu z(t,x)=-\alpha(x) z_{2k+3}(x), \qquad~ \hbox{in } [t_{k+1},t_{k+2}]\times \Gamma_1\\
&z(t_{k+1}) = z_{2k+2},\: \partial_t z(t_{k+1}) = z_{2k+3},\qquad ~ \qquad\hbox{in } \Omega.
\end{aligned}
\right.
$$
Since $z_{2k+2}=z(t_{k+1})\in H^2(\Omega)\cap H^1_{\Gamma_0}(\Omega)$ and $z_{2k+3}=\partial_t z(t_{k+1})\in H^1_{\Gamma_0}(\Omega)$ by induction assumption, $-\alpha z_{2k+3}\vert_{\Gamma_1}\in H^{1/2}(\Gamma_1)$. We can then apply the results of $(i)$ and there exists a unique solution $z \in \mathcal H_{k+1}.$ 

\noindent {\it \underline{Conclusion.}} By induction, for any $k\in\mathbb{N}$, 
$
z \in \mathcal H_k
$
Therefore, from the extension by continuity at the instants $t_k$, one can conclude that \eqref{Eqet} has a unique solution 
$$
z \in C([0,T^*);H^2(\Omega)\cap H^1_{\Gamma_0}(\Omega))\cap C^1([0,T^*);H^1_{\Gamma_0}(\Omega)).
$$\\
\noindent $\bullet$ {\it Avoiding Zeno phenomenon:} \\
The goal is to prove that considering system \eqref{Eqet} subjected to the triggering law \eqref{et-law}, there cannot be an infinite amount of updates in finite time. We will more specifically prove that $T^*$, the maximal time of existence of solution for the closed-loop system defined in (\ref{T}), cannot be finite. The contrary would mean that the sequence $(t_k)_{k\in \mathbb N}$ has an accumulation point, leading to Zeno behavior.
 
 Let us assume that $\lim_{k\to +\infty} t_k= T^*< +\infty$. 
The proof relies on the continuity of $t\mapsto \partial_t z(t,\cdot)$ as a function from $[0,T^*)$ to $H^1_{\Gamma_0}(\Omega)$, that we extend by continuity at $T^*$: $\partial_t z(T^*,\cdot)=\lim_{t\to T^*}\partial_t z(t,\cdot)$.
A trace theorem allows to deduce that  $t\mapsto \partial_t z(t,\cdot)$ is uniformly continuous from the compact set $[0,T^*]$ to the Hilbert space $L^2(\Gamma_1)$. 


The contrapositive of the definition of this uniform continuity brings that for all
$\eta >0$, there exists $\tau>0$ such that, for all $ s,t \in [0,T^*]$
$$
\|\partial_t z(t) - \partial_t z(s)\|_{L^2(\Gamma_1)} > \eta
\quad \Rightarrow \quad
\vert t-s \vert > \tau.
$$
Hence, choosing $k\in \mathbb N$ and  applying this property to $s = t_k$, $t= t_{k+1}$, we ultimately need to prove that we can bound from below, independently of $k$, the quantity 
$\|\partial_t z(t_{k+1}) - \partial_t z(t_k)\|_{L^2(\Gamma_1)} $ in order to avoid the Zeno phenomenon. 
 
And indeed, by definition of $t_{k+1}$ in the event-triggering rule \eqref{et-law}, we have
 $$\|\partial_t z(t_{k+1}) - \partial_t z(t_k)\|^2_{L^2(\Gamma_1)}  
	\geq \gamma E(t_{k+1}) +  \nu_0 \geq  \nu_0  > 0$$
so that, by uniform continuity
$\vert t_{k+1}-t_k \vert > \tau,$
bringing a contradiction. \\
It allows to conclude that the Zeno behavior is avoided. \\

\noindent $\bullet$ {\it Conclusion:}\\
We just proved that $\lim_{k\to +\infty} t_k= T^* = +\infty$, bringing 
$$z \in C^0(\mathbb{R}_+;H^2(\Omega)\cap H^1_{\Gamma_0}(\Omega))\cap C^1(\mathbb{R}_+;H^1_{\Gamma_0}(\Omega)),$$ allowing to end the proof of Theorem~\ref{WP}.
\end{proof}

\subsection{Closed-loop Stability}
\label{sec_exp}

In this section, we prove  Theorem~\ref{Att}, allowing to certify item $(ii)$ of 
Problem~\ref{pb1}.
 In particular, we will exhibit the link between the energy of the system, denoted $E(t)$ and the Lyapunov functional considered in \eqref{Lyap} to address the closed-loop stability.

\begin{proof} [Proof of Theorem \ref{Att}]

Besides the energy $E(t)$ given in \eqref{energy}, let us define
\begin{equation}\label{rho}
\rho(t) = \rho(z(t),\partial_t z(t)) := \int_{\Omega} \left( 2(x-x_0)\cdot \nabla z(t,x) + (n-1)z(t,x) \right) \partial_t z(t,x) \,dx
\end{equation}
so that the Lyapunov functional candidate $V(t)$ given in \eqref{Lyap} satisfies :  
$$V(t)= E(t) + \eps\rho(t).$$
These functionals $E$, $\rho$ and $V$ are all defined on the state trajectories of \eqref{Eqet} that satisfy  (accordingly to Theorem~\ref{WP}) 
$$ \begin{pmatrix} z(t) \\\partial_t z(t) \end{pmatrix}  \in  H^2(\Omega)\cap H^1_{\Gamma_0}(\Omega) \times H^1_{\Gamma_0}(\Omega) $$ 
and were all denoted $F(t)$ instead of $F(z(t), \partial_t z(t))$ in sake of simplicity. \\

\noindent $\bullet$ {\it First Step:} 
Let us prove that there exists a constant $\widetilde C>0$ such that for all $t\geq0$ and $\eps >0$,
\begin{equation}\label{energy2}
\vert V(t) - E(t) \vert \leq \widetilde C \eps E(t).
\end{equation}
As a consequence, the energy $E$ of the system and the proposed Lyapunov functional~$V$ will be equivalent if one chooses $\varepsilon <  1/{\widetilde C}$.\\
Recall that $R = \max_{x\in\Gamma_1} \vert x-x_0 \vert $. 
It is indeed easy to prove, using Cauchy-Schwarz and Poincar\'e 's inequalities recalled in appendix, that 
\begin{eqnarray*}
\eps^{-1} \vert V(t) - E(t) \vert  = \vert \rho(t) \vert  
&=& \left\vert \int_{\Omega} \left( 2(x-x_0)\cdot \nabla z(t,x) + (n-1)z(t,x) \right) \partial_t z(t,x) \,dx\right\vert \\
&\leq& \|\partial_t z(t)\| \left( 2R \|\nabla z(t)\| + (n-1)\|z(t)\|\right)\\
&\leq& \left( 2R + (n-1)C_\Omega \right) \|\partial_t z(t)\| \|\nabla z(t)\|\\
&\leq& \left( 2R + (n-1)C_\Omega \right) E(t).
\end{eqnarray*}
Hence the result \eqref{energy2} for a constant $\widetilde C =   2R + (n-1)C_\Omega $, using  Young's inequality for the last estimate (given in Lemma \ref{Y}) with $\eta = 1$ and $a=\Vert \partial_t z(t)\Vert_{L^2(\Omega)}$ and $b=\Vert \nabla z(t)\Vert_{L^2(\Omega)}$, bringing finally
\begin{equation}\label{energy3}
\left(1-  2R\eps - (n-1)\eps C_\Omega \right) E(t) \leq V(t) \leq \left(1+   2R\eps + (n-1)\eps C_\Omega \right) E(t).
\end{equation}
Thus, by choosing $\varepsilon$ such that $0<\varepsilon < \dfrac 1{2R + (n-1)C_\Omega}= 1/{\widetilde C}$, we are able to guarantee that inequality (\ref{energy3}) holds and reads
\begin{equation}\label{energy3bis}
0 < (1 - \varepsilon \widetilde C) E(t) \leq V(t) \leq (1+\varepsilon  \widetilde C) E(t).
\end{equation}

\noindent $\bullet$ {\it Second Step:} 

The goal of this step is to compute and estimate the time derivative of the functional $V$ along the trajectories of the closed-loop system \eqref{Eqet} with the event-triggering mechanism \eqref{et-law} that reads 
\begin{equation}\label{event-triggered-wave}
	\left\{
	\begin{array}{ll}
		\partial_t^2z(t,x) - \Delta z(t,x) = 0,& \hbox{in }\mathbb{R}_+\times \Omega\\
		z(t,x) = 0, & \hbox{on }\mathbb{R}_+\times \Gamma_0\\
		\partial_\nu z(t,x) = -\alpha(x)\partial_t z(t,x) + \alpha(x) e_k(t,x),  &  \hbox{on }[t_k,t_{k+1}[\times \Gamma_1, 
		\quad  \forall k\in\mathbb N \\
		z(0,x)=z_0(x),~\partial_t z(0,x)=z_1(x),&    \hbox{in } \Omega,
	\end{array}
	\right.
\end{equation}
where we used 
\begin{equation}\label{deviation}
	e_k(t,\cdot):=\partial_t z(t,\cdot)-\partial_t z(t_k,\cdot),~ \hbox{ on } \Gamma_1, ~ \forall t\in [t_k,t_{k+1}[.
\end{equation}
Using integrations by parts, boundary conditions and the definition of $\alpha$ given in \eqref{alpha}
we obtain, for all $t\geq0$, satisfying $t \in [t_k,t_{k+1}[$, on the one hand, since the function $z$ is sufficiently regular
\begin{eqnarray*}
&&\dot E(t) = \frac{1}{2} \dfrac d{dt} \left( \int_\Omega \vert \partial_t z(t,x) \vert ^2 \,dx +  \int_\Omega \vert \nabla z(t,x) \vert ^2 \,dx\right)\\
&&= \int_\Omega \partial_t^2 z(t,x)\partial_t z(t,x)\,dx + \int_\Omega \partial_t\nabla  z(t,x)\cdot \nabla  z(t,x) \,dx \\
&&=\int_\Omega \left( \partial_t^2 z(t,x) - \Delta z(t,x)\right) \partial_t z(t,x)\,dx + \int_{\Gamma_1} \partial_t z(t,x) \partial_\nu  z(t,x) \,d\sigma \\
&&=  - \alpha_1\int_{\Gamma_1}\vert \partial_t z(t,x) \vert ^2\, (x-x_0)\cdot \nu(x) \,d\sigma+ \alpha_1 \int_{\Gamma_1} e_k(t,x) \partial_t z(t,x)  \,(x-x_0)\cdot \nu(x) \,d\sigma
\end{eqnarray*}
and on the other hand
\begin{eqnarray*}
\dot \rho(t) &=&  \dfrac d{dt} \left( \int_{\Omega} \left( 2(x-x_0)\cdot \nabla z(t,x) + (n-1)z(t,x) \right) \partial_t z(t,x)  \right)\\
& = & \int_{\Omega} \left( 2(x-x_0)\cdot \nabla z(t,x)\right)  \Delta z(t,x) \,dx  +  \int_{\Omega} (n-1)z(t,x)  \Delta z(t,x) \,dx \\
&&+\int_{\Omega}(x-x_0)\cdot \nabla (\vert \partial_t z(t,x) \vert ^2 ) \,dx  +\int_{\Omega} (n-1) \vert \partial_t z(t,x) \vert ^2  \,dx\\
& = & 2 \int_{\Omega}  \Delta z(t,x) (x-x_0)\cdot \nabla z(t,x) \,dx   - \int_{\Omega} (n-1) \vert \nabla z(t,x) \vert ^2 \,dx \\
&&+ \int_{\partial\Omega} (n-1)z(t,x)  \nabla z(t,x)\cdot\nu(x) \,d\sigma - \int_{\Omega}n \vert \partial_t z(t,x) \vert ^2\,dx \\
&& +\int_{\partial\Omega} \vert \partial_t z(t,x) \vert ^2 (x-x_0)\cdot\nu(x)  \,d\sigma +\int_{\Omega} (n-1) \vert \partial_t z(t,x) \vert ^2 \,dx.  
\end{eqnarray*}
Let us now calculate in detail the first term of this last expression. It is a result known as the Rellich identity  \cite{komornik1991rapid} (also see \cite[Lemma 7.6.3]{tucsnak2009}) that reads, for $v\in H^2$,
\begin{multline*}
2 \int_{\Omega}  \Delta v(x)(x-x_0)\cdot \nabla v(x) \,dx   
=(n-2)\int_{\Omega} \vert \nabla v(x) \vert ^2 \,dx \\
+ 2 \int_{\partial\Omega} (x-x_0)\cdot \nabla v(x)\,  \partial_\nu v(x)\,d\sigma 
- \int_{\partial\Omega}  \vert \nabla v(x) \vert ^2 (x-x_0)\cdot \nu(x) \,d\sigma
\end{multline*}
and is proved by integration by parts. Note that some estimate of the same shape can also be obtain under less restrictive hypothesis (see e.g., \cite{komornik1991rapid}). We apply it to $v=z(t)$.

Therefore we get (omitting the $x$ variable to lighten the writing)
\begin{multline*}
\dot \rho(t) = - \int_{\Omega} \vert \nabla z(t) \vert ^2  - \int_{\Omega}\vert \partial_t z(t) \vert ^2
 + 2 \int_{\partial\Omega} (x-x_0)\cdot \nabla z(t) \, \partial_\nu z(t)\\
- \int_{\partial\Omega}  \vert \nabla z(t) \vert ^2 (x-x_0)\cdot \nu  
+ \int_{\partial\Omega} (n-1)z(t)   \partial_\nu z(t) 
 +\int_{\partial\Omega} \vert \partial_t z(t) \vert ^2 (x-x_0)\cdot\nu.
 \end{multline*}
Moreover, one should remember that for any $t\geq 0$,
\begin{itemize}
\item[-] since $z (t,x)= 0$ on $\Gamma_0$,  one can deduce that for all $x\in \Gamma_0$, $\partial_t z (x,t)= 0$ and $\nabla z(x,t) = \partial_\nu z(x,t) \, \nu(x)$ so that
 $\nabla z(x,t) \cdot (x-x_0) = \partial_\nu z(x,t) \,(x-x_0)\cdot \nu(x)$
\item[-] on $\Gamma_1$, $\partial_\nu z(x,t) = \left( - \partial_t z(x,t) + e_k(x,t) \right)\alpha_1(x-x_0)\cdot \nu (x)$.
\end{itemize}
Thus from these assumptions on  $\Gamma_0$ and  $\Gamma_1$, one gets
\begin{eqnarray*}
\dot \rho(t) 
&=& - \int_{\Omega} \vert \nabla z(t) \vert ^2  - \int_{\Omega}\vert \partial_t z(t) \vert ^2 
 +~ 2\alpha_1 \int_{\Gamma_1}(x-x_0)\cdot \nabla z(t) \big(e_k(t) - \partial_t z(t)\big) (x-x_0)\cdot\nu  \\
&&+~ 2 \int_{\Gamma_0}  \vert \partial_\nu z(t) \vert ^2 (x-x_0)\cdot \nu - \int_{\Gamma_0}  \vert \nabla z(t) \vert ^2 (x-x_0)\cdot \nu
- \int_{\Gamma_1}  \vert \nabla z(t) \vert ^2 (x-x_0)\cdot \nu\\
&&+\int_{\Gamma_1} \vert \partial_t z(t) \vert ^2 (x-x_0)\cdot\nu
+~\alpha_1 \int_{\Gamma_1} (n-1)z(t) \big( e_k(t)- \partial_t z(t)\big) (x-x_0) \cdot\nu.
 \end{eqnarray*}
 Since $\vert \nabla z \vert ^2= \vert \partial_\nu z\vert^2$ on $\Gamma_0$, and from the definition of $\Gamma_1 = \partial \Omega \setminus \Gamma_0$, the boundary terms on $\Gamma_0$ gather into one negative term 
 $$
 \int_{\Gamma_0}  \vert \partial_\nu z(t) \vert ^2 (x-x_0)\cdot \nu \leq 0.
 $$
Gathering the estimates of  $\dot E(t)$ and $\dot \rho(t)$  one obtains, for $\dot V(t) =  \dot E(t) + \eps \dot \rho(t)$,
\begin{eqnarray*}
\dot V(t) 
&\leq& - ~\eps\int_{\Omega} \vert \nabla z(t) \vert ^2  - \eps \int_{\Omega}\vert \partial_t z(t) \vert ^2 
+ (\eps - \alpha_1) \int_{\Gamma_1}\vert \partial_t z(t) \vert ^2\, (x-x_0)\cdot \nu \\
&&+ ~ \alpha_1\int_{\Gamma_1} e_k(t) \partial_t z(t)  \,(x-x_0)\cdot \nu
- \eps \int_{\Gamma_1}  \vert \nabla z(t) \vert ^2 (x-x_0)\cdot \nu\\
&& - ~2\alpha_1\eps \int_{\Gamma_1}(x-x_0)\cdot \nabla z(t) \, \partial_t z(t)\, (x-x_0)\cdot\nu  
 + 2\alpha_1\eps \int_{\Gamma_1}(x-x_0)\cdot \nabla z(t) \, e_k(t)\, (x-x_0)\cdot\nu  \\
&&- ~\alpha_1\eps\int_{\Gamma_1} (n-1)z(t)  \partial_t z(t)(x-x_0)\cdot\nu
+\alpha_1\eps \int_{\Gamma_1} (n-1)z(t)  e_k(t) (x-x_0) \cdot\nu.
\end{eqnarray*}
Since $z(t) \in H^1_{\Gamma_0}(\Omega)$, the trace estimate of Lemma~\ref{trace}, given and proved in appendix, implies that, for any $\lambda_2 >0$, with $\beta = RC_\Omega + nC_\Omega^2$
\begin{equation}\label{L1}
- \lambda_2 \int_{ \Gamma_1} \vert z(t) \vert ^2 (x-x_0)\cdot \nu + \lambda_2 \beta \int_{\Omega} \vert \nabla z(t) \vert ^2 \geq 0.
\end{equation}
Recall also that the update instants $t_k$ follow the event-triggering law \eqref{et-law}, so that for any $\lambda_1 >0$ 
\begin{equation} \label{L2} 
- \dfrac{\lambda_1}R \int_{ \Gamma_1} \vert e_k(t) \vert ^2 (x-x_0)\cdot \nu  
+ \dfrac{\gamma \lambda_1}2 \int_{\Omega} \vert \nabla z(t) \vert ^2 
 + \dfrac{\gamma \lambda_1}2 \int_{\Omega} \vert \partial_t z(t) \vert ^2  + \lambda_1\nu_0 \geq 0 . 
 \end{equation}
 Therefore, adding \eqref{L1} and \eqref{L2} to the right-hand side of the last estimate on $\dot V(t) $, it  reads now
\begin{eqnarray*}
\dot V(t) &\leq&  \left(- \eps + \dfrac{\gamma \lambda_1}2  + \lambda_2 \beta  \right)\int_{\Omega} \vert \nabla z(t) \vert ^2 
+  \left(- \eps + \dfrac{\gamma \lambda_1}2  \right) \int_{\Omega}\vert \partial_t z(t) \vert ^2 
 - \lambda_2 \int_{ \Gamma_1} \vert z(t) \vert ^2 (x-x_0)\cdot \nu \\
&&-~ \dfrac\eps{R^2} \int_{\Gamma_1}  \vert (x-x_0)\cdot\nabla z(t) \vert ^2 (x-x_0)\cdot \nu
- ~\dfrac{\lambda_1}R \int_{ \Gamma_1} \vert e_k(t) \vert ^2 (x-x_0)\cdot \nu  \\
&&+~ (\eps - \alpha_1) \int_{\Gamma_1}\vert \partial_t z(t) \vert ^2\, (x-x_0)\cdot \nu + ~\alpha_1 \int_{\Gamma_1} e_k(t) \partial_t z(t)  \,(x-x_0)\cdot \nu\\
&& - ~2\alpha_1\eps \int_{\Gamma_1}(x-x_0)\cdot \nabla z(t) \, \partial_t z(t)\, (x-x_0)\cdot\nu  
 + ~2\alpha_1\eps \int_{\Gamma_1}(x-x_0)\cdot \nabla z(t) \, e_k(t)\, (x-x_0)\cdot\nu  \\
&&- ~\eps (n-1)\alpha_1\int_{\Gamma_1} z(t)  \partial_t z(t)(x-x_0)\cdot\nu
+~\eps (n-1)\alpha_1 \int_{\Gamma_1} z(t)  e_k(t) (x-x_0) \cdot\nu
 + \lambda_1 \nu_0 .
\end{eqnarray*}
Let us define the boundary trace of an augmented state by, for all $(t,x)\in \mathbb R_+ \times \Gamma_1$,
$$
\xi(t,x) =\left( z(t,x), (x-x_0) \cdot \nabla z(t,x), \partial_t z(t,x) , e_k(t,x) \right)^\top
$$
 and notice that the regularity $ z \in C^0(\mathbb R_+; H^2(\Omega)\cap H^1_{\Gamma_0}(\Omega)) \cap C^1(\mathbb R_+; H^1_{\Gamma_0}(\Omega)) $ of the state $z$ brings $\xi(t) \in L^2(\Gamma_1)^4$ since $H^1_{\Gamma_0}(\Omega) \hookrightarrow H^{\frac 12}(\partial\Omega) \subset L^2(\partial\Omega)$.
 Using also the symmetric matrix 
$$ M = \left(\begin{array}{cccc}
 -\lambda_2 & 0 & -  (n-1)\alpha_1\eps/2 & (n-1)\alpha_1\eps/2 \\
 * & - \eps/R^2 & - \alpha_1\eps &  \alpha_1\eps \\
 * & * & \eps-\alpha_1  &  \alpha_1/2 \\
 * & * & * &  -\lambda_1/R
\end{array}\right)
$$
it allows to write finally that
\begin{equation}\label{V1}
\dot V(t) 
\leq 2 \left( -\eps  +\dfrac{\lambda_1 \gamma}2 + \lambda_2 C_\Omega (R + nC_\Omega) \right) E(t)
+  \int_{\Gamma_1}\xi(t)^\top M \xi(t) (x-x_0) \cdot\nu
 + \lambda_1 \nu_0.
\end{equation}

\noindent $\bullet$ {\it Third Step:} 
The final idea is to ensure that one gets $\dot V(t) < - 2 \delta V(t)$ when the solutions 
$(z,\partial_t z)$ of the closed loop evolve outside the attractor, that is, evolve in $H^1_{\Gamma_0}(\Omega) \times L^2(\Omega)
 \setminus \mathcal{A}$. 
We will characterize an outer-approximation of the true attractor, and thus consider that $\mathcal{A}$ is defined from the Lyapunov functional $V$ as: 
$$
\mathcal{A} =\{(v,w) \in H^1_{\Gamma_0}(\Omega)\times L^2(\Omega); V(v,w) \leq \bar{r}\},
$$
where we recall that 
$$
V(v,w) = \frac{1}{2}\int_\Omega \vert w\vert ^2  + \frac{1}{2}\int_\Omega \vert \nabla v \vert ^2 \\
+ \eps\int_{\Omega} \left( 2(x-x_0)\cdot \nabla v + (n-1)v \right) w.
$$
%
 Let us first calculate, from \eqref{V1}, that provided assumption \eqref{eq:pscond2a} holds, and using \eqref{energy3bis}:
  \begin{multline}\label{V2}
\dot V(t) +2 \delta V(t) 
\leq 2 \left( \dfrac{ -\eps  +\dfrac{\lambda_1 \gamma}2 + \lambda_2 C_\Omega (R + nC_\Omega)}{1+\varepsilon  (2R + (n-1)C_\Omega)} + \delta \right) V(t) \\
+  \int_{\Gamma_1}\xi(t)^\top M \xi(t) (x-x_0) \cdot\nu
 + \lambda_1 \nu_0.
\end{multline}
One should notice now that the right-hand side of \eqref{V2} is negative as soon as  
\begin{eqnarray*}
&& \dfrac{ -\eps  +\dfrac{\lambda_1 \gamma}2 + \lambda_2 C_\Omega (R + nC_\Omega)}{1+\varepsilon  (2R + (n-1)C_\Omega)} + \delta < 0\\
&&~\\
 &&M \prec 0\\
&& V(t) \geq \dfrac{\lambda_1 \nu_0}{2} \left(  \dfrac{ -\eps  +\dfrac{\lambda_1 \gamma}2 + \lambda_2 C_\Omega (R + nC_\Omega)}{1+\varepsilon  (2R + (n-1)C_\Omega)} + \delta   \right)^{-1}.
 \end{eqnarray*}
 Thus, under assumptions \eqref{eq:pscond2a}, \eqref{eq:pscond1a} and \eqref{delta} of Theorem~\ref{Att}, if the trajectory $(z(t),\partial_t z(t))$ evolves outside $\mathcal A$, meaning that if $V(t) \geq r$, then one has 
 $$\dot V(t) \leq - 2 \delta V(t).$$
Therefore, we have proved that for any initial condition $(z_0,z_1)$ in the functional space $H^2(\Omega)\cap H^1_{\Gamma_0}(\Omega)\times H^1_{\Gamma_0}(\Omega)$, the closed-loop trajectories 
$  \begin{pmatrix} z \\\partial_t z \end{pmatrix}  \in C(\mathbb R_+; H^2(\Omega)\cap H^1_{\Gamma_0}(\Omega)) \times C(\mathbb R_+; H^1_{\Gamma_0}(\Omega)) $
converge exponentially fast towards the attractor 
	$$\mathcal{A} = \left\{\begin{pmatrix} v \\w \end{pmatrix} \in  H^1_{\Gamma_0}(\Omega)\times L^2(\Omega) \hbox{ such that } V(v,w) < r \right\}. 
$$
And we conclude here the proof of Theorem~\ref{Att}.
\end{proof}

\subsection{Feasibility of the theorem's assumptions.}
\label{sec_f}

Let us finally  prove that the conditions of Theorem \ref{Att} have always a feasible solution $(\eps,\lambda_1,\lambda_2)$.

\begin{theorem}\label{GES2}
Let $\Omega \subset \mathbb R^n$, $R = \max_{x\in\Gamma_1} \vert x-x_0 \vert $, $\alpha_1>0$ and $C_\Omega$ the Poincar\'e constant of domain $\Omega$.
Let $\gamma>0$ and $\nu_0$ be the tuning parameters of \eqref{et-law}.
There always exists coefficients $\lambda_1> 0$, $\lambda_2 > 0$ and $\eps>0$ such that $\varepsilon <  1/(2R + (n-1)C_\Omega)$ and conditions \eqref{eq:pscond2a} (or \eqref{eq:pscond2}) and  \eqref{eq:pscond1a} hold.
Therefore, the conclusion of Theorem~\ref{Att} holds as well.
\end{theorem}
\begin{proof} 
Our goal here is to prove that the assumptions of Theorem~\ref{Att} cannot form an empty set. 
To start with, having set the tuning parameters $\gamma>0$ and $\nu_0  >0$ of \eqref{et-law}, one is looking for positive constants  $\lambda_1$, $\lambda_2$ and $\varepsilon <  1/(2R + (n-1)C_\Omega)$  such that  
$$ M \prec 0 ~~\hbox{ and }~~  -\eps  +\dfrac{\lambda_1 \gamma}2 + \lambda_2 C_\Omega(R + nC_\Omega) < 0.$$
Let us denote  $a =(n-1)/2$ and write 
$
M= M_0 -\lambda_2 B_0B_0^\top
$ as 
$$ M = \left(\begin{array}{cccc}
 0 & 0 & -  a\alpha_1\eps & a\alpha_1\eps \\
 0 & - \eps/R^2 & - \alpha_1\eps &  \alpha_1\eps \\
 -  a\alpha_1\eps & - \alpha_1\eps & \eps-\alpha_1  & \alpha_1/2 \\
  a\alpha_1\eps &  \alpha_1\eps & \alpha_1/2 &  -\lambda_1/R
\end{array}\right) 
- \left(\begin{array}{c}1 \\0 \\0 \\0\end{array}\right)
\lambda_2
\left(\begin{array}{cccc}1 & 0 & 0 & 0\end{array}\right).
$$
Using Finsler's Lemma (see e.g. \cite{oli:ske2001}, \cite{EPA-book} and recalled in appendix) we have the equivalence between 
$M \prec 0$ and 
 $X^\top M_0 X < 0$ for all  $ X \in \ker B_0 \setminus \{0\}$.
Thus, denoting $$N_0 =  \left(\begin{array}{ccc}
 0 & 0 & 0  \\
 1 & 0& 0  \\
 0 & 1 & 0   \\
  0 & 0  & 1 
\end{array}\right),  \hbox{ satisfying } N_0^\top B_0 =  \left(\begin{array}{c}0 \\0 \\0 \end{array}\right)
 \hbox{ and  } B_0^\top N_0 = \left(\begin{array}{ccc}  0 & 0 & 0\end{array}\right),
$$
it is equivalent with 
$$
M_1 = N_0^\top M N_0 = N_0^\top M_0 N_0 = 
\left(\begin{array}{ccc}
 - \eps/R^2 & - \alpha_1\eps &  \alpha_1\eps \\
 - \alpha_1\eps & \eps-\alpha_1  &  \alpha_1/2 \\
  \alpha_1\eps & \alpha_1/2 &  -\lambda_1/R
\end{array}\right) 
\prec 0.
$$ 
Again, we can write $M_1 = M_2 - \dfrac{\lambda_1}R B_2 B_2^\top$ as 
$$
M_1 =  
\left(\begin{array}{ccc}
 - \eps/R^2 & - \alpha_1\eps &  \alpha_1\eps \\
 -\alpha_1 \eps & \eps-\alpha_1  &  \alpha_1/2 \\
  \alpha_1\eps & \alpha_1/2 &  0
\end{array}\right) 
- \left(\begin{array}{c}0 \\0 \\1
\end{array}\right)
\dfrac{\lambda_1}R
\left(\begin{array}{cccc}0 & 0 & 1\end{array}\right).
$$ 
Using again Finsler's Lemma we have the equivalence between 
$M_1 \prec 0$ and 
 $X^\top M_2 X < 0$ for all  $ X \in \ker B_2 \setminus \{0\}$.
Denoting 
$$N_2 =  \left(\begin{array}{cc}
 1& 0   \\
 0 & 1  \\
 0 & 0   
\end{array}\right),  \hbox{ satisfying } N_2^\top B_2 =  \left(\begin{array}{c}0 \\0  \end{array}\right)
 \hbox{ and  } B_2^\top N_2 = \left(\begin{array}{cc}  0 & 0 \end{array}\right),
$$
it is equivalent with 
$$
M_3 = N_2^\top M_1 N_2 = N_2^\top M_2 N_2 = 
\left(\begin{array}{cc}
 - \eps/R^2 & - \alpha_1\eps  \\
 - \alpha_1\eps & \eps-\alpha_1  
\end{array}\right) 
\prec 0.
$$
Such an LMI is now easy to check : $M_3$ is definite negative as soon as $\eps > 0$ and 
$\eps -\alpha_1 + \alpha_1^2\eps  R^2 < 0$ (for instance using the Schur complement).
Thus, we have proved that there exist $\lambda_1 \geq 0$ and $\lambda_2 \geq 0$ and 
$\eps$ such that 
$$
0 < \eps < \dfrac {\alpha_1}{1+\alpha_1^2R^2}
$$
give a feasible solution to the LMI assumption $M \prec 0$. Besides, remember that we also have to satisfy
$$\varepsilon <  \dfrac1{2R + (n-1)C_\Omega}.$$
Finally, we only have to make sure that it is still possible to have
$$  \dfrac{\lambda_1 \gamma}2 + \lambda_2 C_\Omega(R + nC_\Omega) <  \eps$$
 and choosing $\lambda_1$ and $\lambda_2$ small enough will be necessary and sufficient, ending the solvability of assumptions of Theorem~\ref{Att}.
\end{proof}

\section{Numerical illustration and conclusion}
\label{sec_num}

We illustrate the efficiency of the event-triggering mechanism proposed in this paper by considering the example of a
one-dimensional wave equation set on  $\Omega =(0,\pi)$ with Neumann boundary control $\partial_\nu z = u$
acting at $\Gamma_1=\left\{\pi\right\}$. We consider the initial position and velocity conditions
$$z_0(x)=\sin\left(\frac{x}{2}\right),\qquad z_1(x)=\sin\left(\frac{3x}{2}\right)$$
and the damping coefficient in the control input $u(t,\pi):=-\alpha(\pi)\partial_t z(t,\pi)$
$$\alpha(\pi)=\alpha_1 (\pi-x_0),\quad\hbox{ with }\alpha_1=0.1 \hbox{ and } x_0=-1.$$ 
We aim at comparing the continuous-in-time version of the closed-loop system with the event-triggered closed-loop version. 
In other words we compare the behavior of system \eqref{EqU}-\eqref{eq:ucontinuous} with the one of system \eqref{Eqet} under the event-triggering rule \eqref{et-law}.\\

Taking $\gamma=0.2$ and $\nu_0=0.1$, a feasible solution to conditions of Theorem~\ref{GES2} is $\lambda_1=0.1$, $\lambda_2=0.01$, $\epsilon=0.08$. \\

Now in order to better understand how the sampling acts on the stability result, we present in Figure~\ref{fig2} the repartition of the updates instants, and in Figure~\ref{fig1}, the evolution of the natural energy $E(t)$ of the closed-loop system \eqref{Eqet}-\eqref{et-law} in the following cases:
\begin {itemize}
\item Under the continuous-in-time control (dark line);
\item Under the event-triggered controller (blue line) with $t_k$ given by the event-triggering rule \eqref{et-law};
\item Under the fixed controller $u(t) = -\alpha_1 (\pi-x_0)z_1(\pi)$ (red line);
\item With the controller $u(t) =  -\alpha_1 (\pi-x_0)\partial_t z(k\tau,\pi)$ when $k\tau<t<(k+1)\tau$ (in green) build with periodic sampling under period $\tau=0.36$ so that the number of updates is the same as the one observed during the time T = 10 when following \eqref{et-law}.
\end{itemize}

\begin{figure}[h!]
        \vspace{-5cm}
    \begin{center}
        \includegraphics[width=0.9\textwidth]{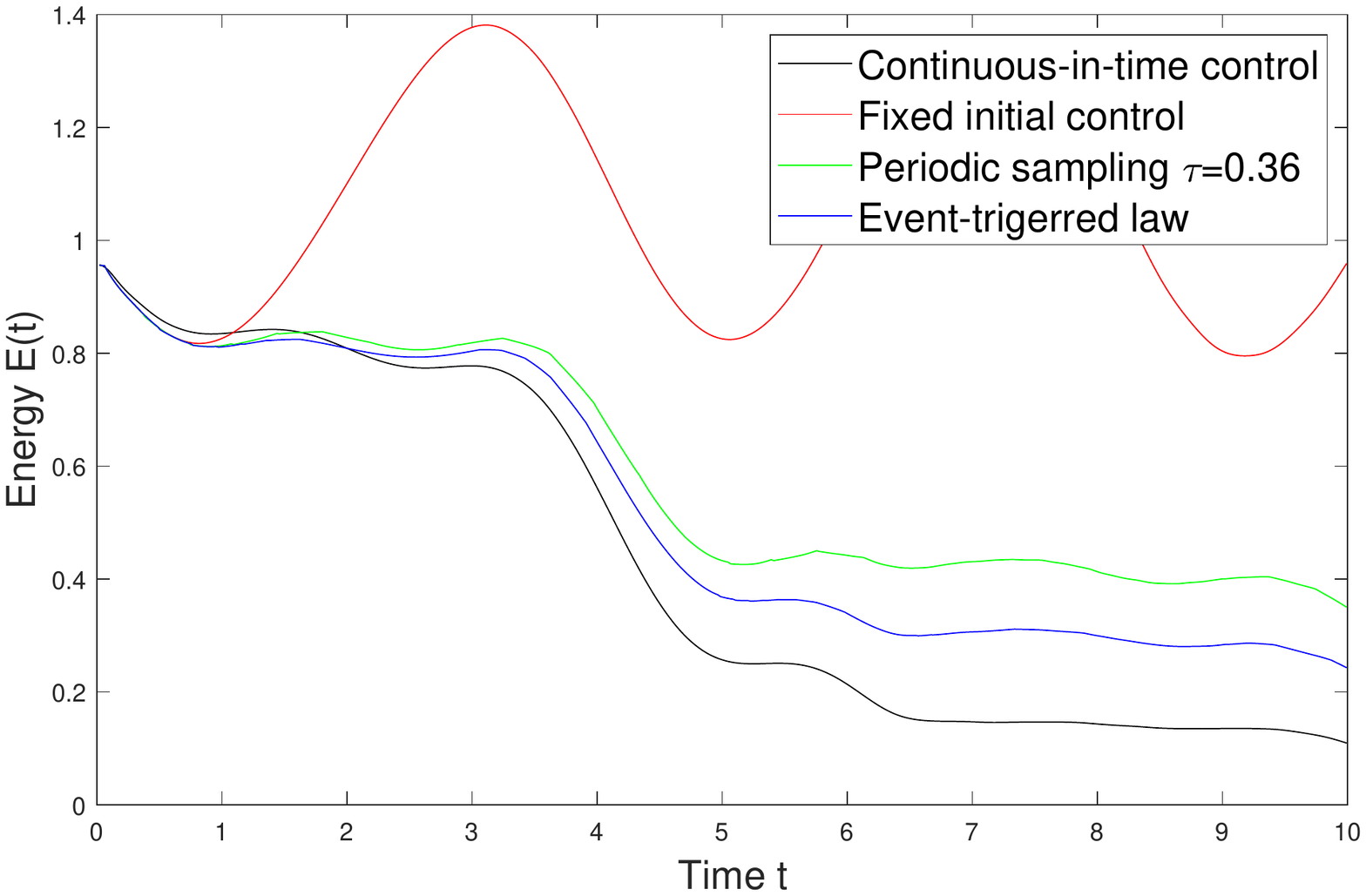}
        \vspace{-6cm}
	\caption{Evolution of the energy $E(t)$ defined in \eqref{energy}.}\label{fig1}
    \end{center}
\end{figure}

Other sampling periods could be tested to compare the efficiency of the two approaches (periodic or event-based) but one should remember that we have proved the exponential convergence towards an attractor only in the event-triggered control context. We actually do not know, in our boundary control setting for the wave equation, any proof of convergence in the periodic setting, even for a small enough sampling period.
Nevertheless, by using $\tau$ to denote the period, our numerical simulation allows to observe that if $0.25\leq\tau\leq0.36$ then the energy decreases but stays  above the one under \eqref{et-law}, and if $\tau\leq0.25$ the energy decreases and stays below. Moreover, if $\tau=0.05$ the evolution of the energy remains very close to the one in the case of a continuous  damping. \\

\begin{figure}[h!]
                  \vspace{-5cm}
  \begin{center}
        \includegraphics[width=0.9\textwidth]{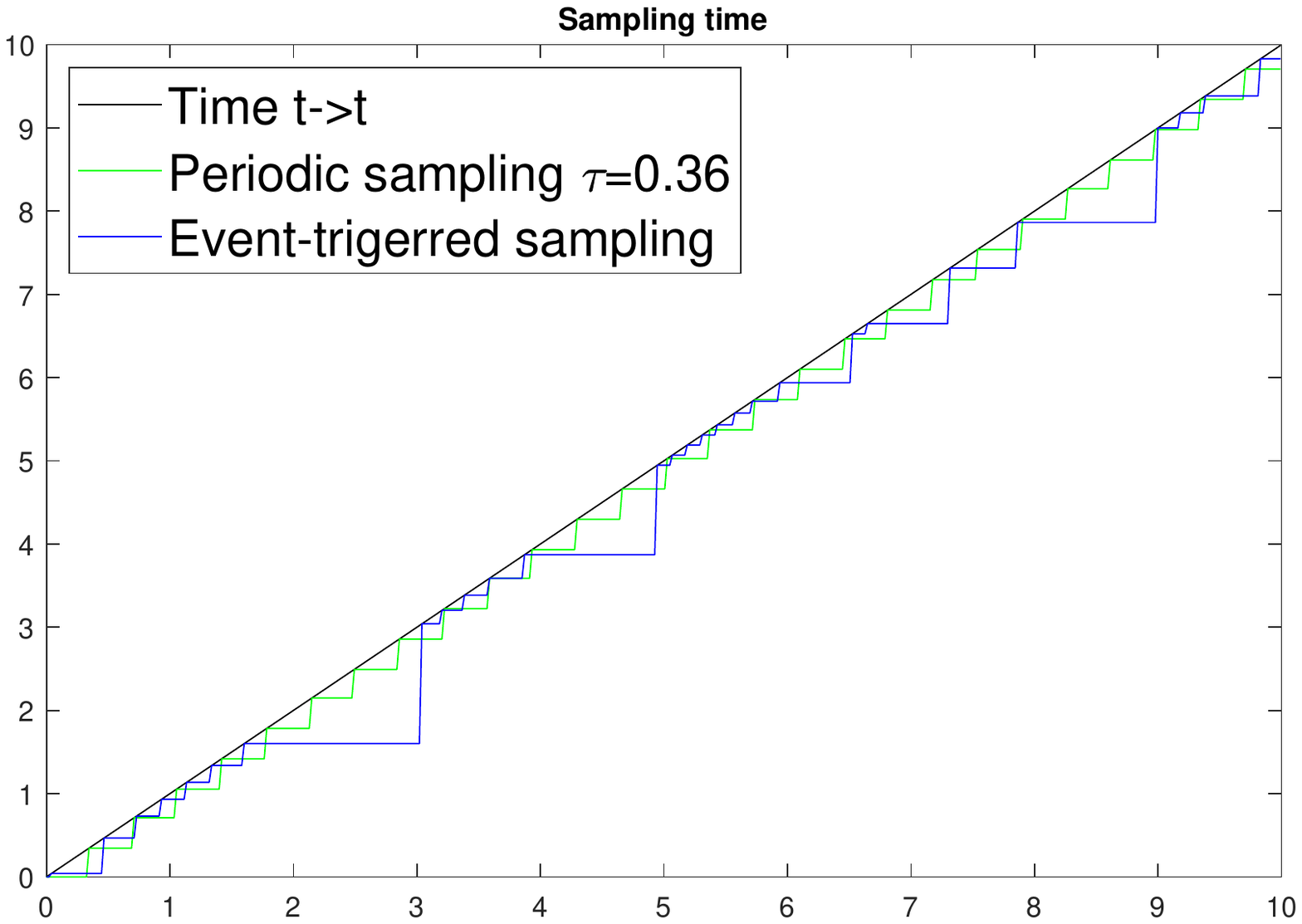}
                \vspace{-6cm}
                \caption{Time evolution of the event-triggering mechanisms \eqref{et-law} and $\tau$-periodic sampling.}\label{fig2}
    \end{center}
\end{figure}

Furthermore, by considering  $\tau=0.36$ or $\tau\geq3$ one can observe on Figure~\ref{fig4} the non-decreasing of the energy corresponding to the repartition of the updates shown in Figure~\ref{fig3}.
%
%
\begin{figure}[p]
                  \vspace{-3cm}
  \begin{center}
        \includegraphics[width=0.6\textwidth]{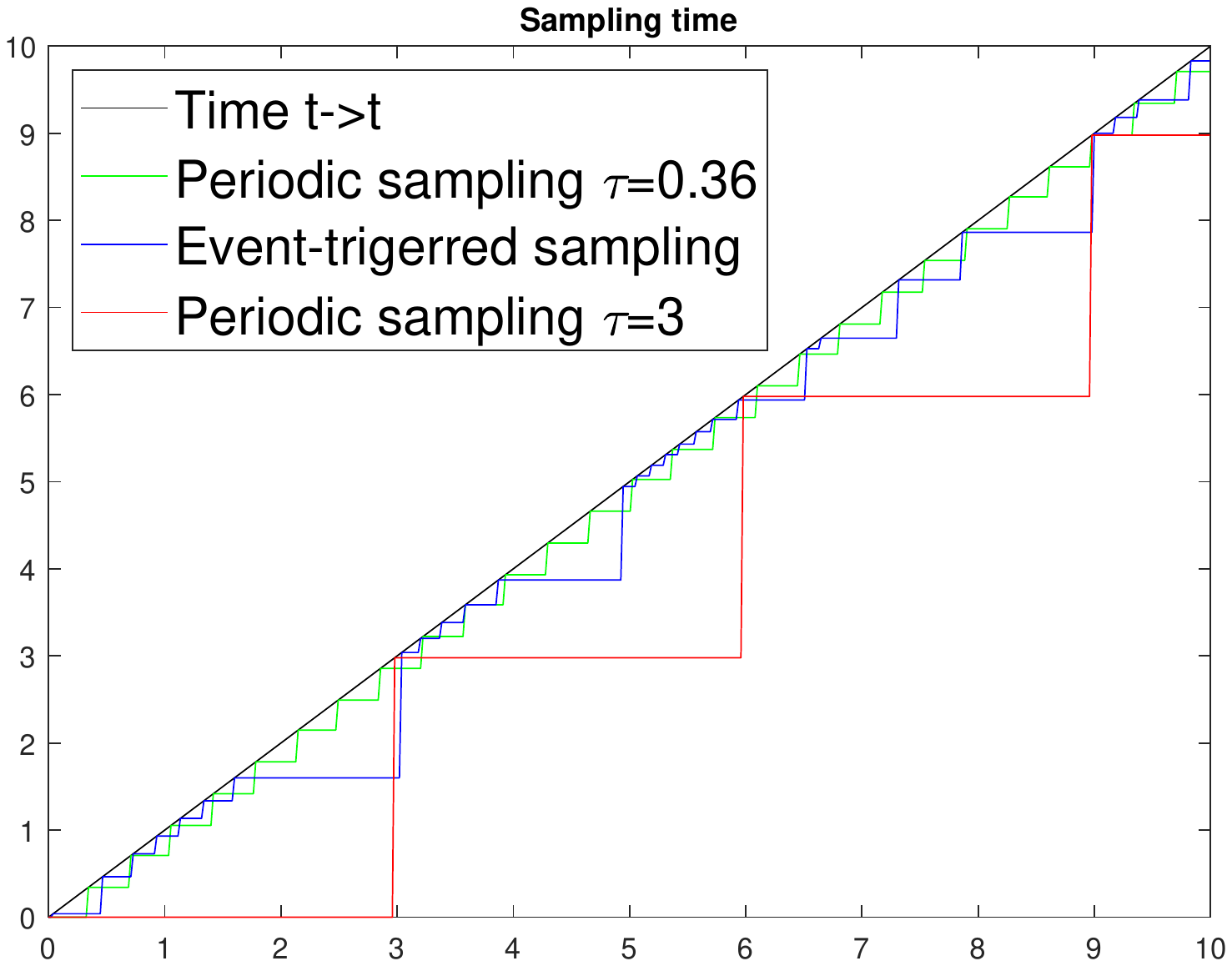}
                \vspace{-4cm}
                \caption{Time evolution of the event-triggering mechanisms \eqref{et-law} and periodic sampling of period $\tau=0.36$ and $\tau=3$.}\label{fig3}
    \end{center}
\end{figure}
\begin{figure}[p]
                  \vspace{-5cm}
    \begin{center}
        \includegraphics[width=0.6\textwidth]{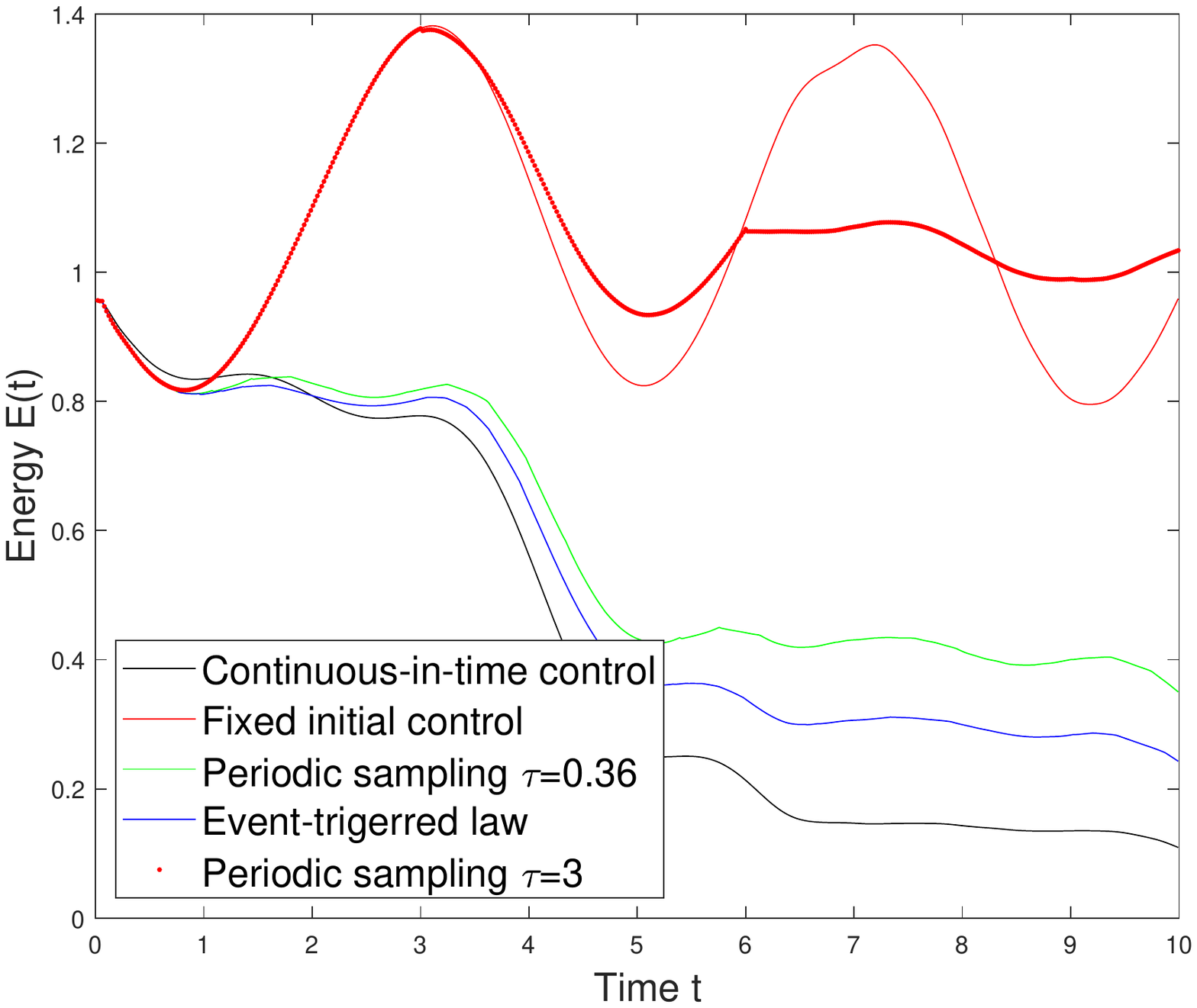}
        \vspace{-3.5cm}
	\caption{Evolution of the energy $E(t)$ defined in \eqref{energy}.}\label{fig4}
    \end{center}
\end{figure}

\subsection{Conclusion}

This paper considered a multi-dimensional wave equation subjected to a boundary event-triggered control. First, proofs of the well-posedness of this system and the absence of Zeno phenomenon were presented. Second,  the exponential convergence towards an attractor containing the origin was ensured, using an appropriate Lyapunov functional based on an LMI condition. Note that an outer-approximation of the attractor built as a level-set of the Lyapunov function was provided. Furthermore, the feasibility of the LMI condition was checked thanks to the classical tools associated to this framework. 

This work opens the doors for other research lines in the field of automatic control theory for infinite dimensional systems. Among them, it  would be interesting to investigate, for example, various effects of digital implementation (or even the consideration of specific constraints in the closed loop) for other multi-dimensional PDEs. Localized or boundary control of PDEs are settings for which challenging mathematical questions are at stake, and questions such as the effect on stability when event-triggering updates laws or input saturations occur are still quite open.

\section{Appendix}
\label{sec_appendix}

\begin{lemma}[Young's inequality]\label{Y}
$$
2ab \leq \eta a^2 + \dfrac{1}{\eta} b^2, \qquad \forall a,b \in \mathbb R, \forall \eta>0.
$$
\end{lemma} 
\begin{lemma}[Finsler's Lemma \cite{oli:ske2001}]\label{finsler}
For given $B\in \mathbb R^{n\times m}$ and $Q \in \mathbb S^n$, the following assertions are equivalent:
\begin{itemize}
\item $
	\exists \mu \in \mathbb R
	\quad \hbox{such that} \quad Q-\mu BB^\top \prec 0$

\item $
	\forall X \in \ker B \setminus \{0\}, \quad X^\top Q X < 0.
$
\end{itemize}
\end{lemma} 
Another version of Finsler Lemma is proposed in \cite{EPA-book} with $\mu >0$ but such a case can be recovered from Lemma \ref{finsler}.

\begin{lemma}[Poincar\'e's inequality]\label{Poincare}
There exists a constant $C_\Omega>0$ depending on the size, the geometry and the regularity of the bounded domain $\Omega$ such that for any function $u \in H^1_{\Gamma_0}(\Omega)$, 
$
\|u\|_{L^2} \leq C_\Omega \|\nabla u\|_{L^2}.
$
\end{lemma} 
The optimal constant $C_\Omega$ is usually called ``Poincar\'e constant" of the domain $\Omega$.
For instance, if $N=1$, Wirtinger's inequality allows to get $C_{[a,b]} =  (b-a)/(2\pi)$. And in general, if $\Omega \subset \mathbb R^N$ is a convex Lipschitzian domain of diameter $d$, then $C_\Omega \leq  d/\pi$. 
\begin{lemma}[A trace inequality]\label{trace}
Assuming 
$\Gamma_0 = \{ x\in \partial\Omega, (x-x_0)\cdot \nu(x) \leq 0\},$
then 
$$
\forall u\in H^1_{\Gamma_0}(\Omega), \quad 
\int_{\partial\Omega \setminus \Gamma_0} \vert u(x) \vert ^2 (x-x_0)\cdot \nu(x)\,d\sigma 
\leq 
\beta \int_{\Omega} \vert \nabla u(x) \vert ^2 \,dx. 
$$
\end{lemma} 
\begin{proof} - 
It stems from Poincar\'e's inequality and a trace estimate in $H^1(\Omega)$. Indeed, an easy calculation with an integration by parts brings 
$$
\int_{\partial\Omega}  \vert u(x) \vert ^2 (x-x_0)\cdot \nu(x) \,d\sigma
= 
2 \int_{\Omega}  u(x) \nabla u(x) \cdot (x-x_0) \,dx + \int_{\Omega} n\vert u \vert^2\,dx.
$$
Then using $R = \max_{x\in\Gamma_1} \vert x-x_0 \vert $, Cauchy-Schwartz and Poincar\'e inequalities, we can write
$$
\int_{\partial\Omega}  \vert u(x) \vert ^2 (x-x_0)\cdot \nu(x) \,d\sigma
\leq 
 R \Vert u\Vert_{L^2} \Vert\nabla u\Vert_{L^2} + n\Vert u\Vert^2_{L^2}
 \leq \beta \Vert \nabla u\Vert^2_{L^2}
$$
with $\beta = RC_\Omega + nC_\Omega^2$. Therefore,  $u=0$ on $\Gamma_0$ allows to conclude the proof.
\end{proof}

\end{document}